\documentclass[12pt]{article}

\usepackage{amsmath, amssymb}
\usepackage{amsthm}
\usepackage[margin=1in]{geometry}
\usepackage{ifpdf}
\ifpdf
  \usepackage[hidelinks]{hyperref}
\else
  \usepackage[dvipdfmx,hidelinks]{hyperref}
\fi

\title{An alternative proof of Miyashita's theorem in a skew polynomial ring II}
\author{Satoshi Yamanaka\\[0.25em]
\small Department of Integrated Science and Technology\\
\small National Institute of Technology, Tsuyama College\\
\small Okayama 708-8509, JAPAN\\
\small \texttt{yamanaka@tsuyama.kosen-ac.jp}}
\date{}

\newtheorem{Theorem}{\quad Theorem}[section]

\newtheorem{Corollary}[Theorem]{\quad Corollary}
\newtheorem{Proposition}[Theorem]{\quad Proposition}
\newtheorem{Lemma}[Theorem]{\quad Lemma}

\allowdisplaybreaks[1]

\begin{document}

\maketitle

\begin{abstract}
Y. Miyashita gave characterizations of a separable polynomial and a Hirata separable polynomial 
in skew polynomial rings. 
In the previous paper, the author and S. Ikehata gave direct and elementary proofs of Miyashita's theorems  
in skew polynomial rings of automorphism type $B[X;\rho]$ and derivation type $B[X;D]$, respectively.  
The purpose of this paper is to give proofs for them in the general skew polynomial ring $B[X;\rho,D]$.
\end{abstract}

\vspace{0.5cm}
\noindent
{\small {\bf Note for the arXiv version.} This manuscript is the author's accepted manuscript of the article ``An Alternative Proof of Miyashita's Theorem in a Skew Polynomial Ring II,'' published in {\it Gulf Journal of Mathematics} {\bf 5} (2017), no.~4, 9--17. The content of this manuscript corresponds to the published article. The version of record is available at \href{https://doi.org/10.56947/gjom.v5i4.116}{10.56947/gjom.v5i4.116}.}
\vspace{0.5cm}



\section{Introduction and Preliminaries}

Throughout this paper, all rings has an identity $1$.  
Let $A/B$ be  a ring extension with common identity $1$.  
$A/B$ is called {\it separable} if the $A$-$A$-homomorphism of $A \otimes_BA$ onto $A$
defined by $a \otimes b \mapsto ab$ splits, and $A/B$ is called {\it Hirata separable} if $A \otimes_BA$ is
$A$-$A$-isomorphic to a direct summand of a finite direct sum of copies of $A$. 
It is well known that 
a Hirata separable extension is  separable. 
%The notion of Hirata separable extensions was
%introduced by K. Hirata as a generalization of Azumaya algebras. 
We denote 
\begin{align*}
(A \otimes_B A)^A &= \{ \mu \in A \otimes_B A \ | \ z \mu = \mu z \ {\rm for \ any} \ z \in A \},\\
V_A(B) &= \{ g \in A \ | \ \alpha g = g \alpha  \ {\rm for \ any} \ \alpha \in B \}.
\end{align*}
%$$
%V_A(B) = \{ x \in A \ | \ xb = bx  \ (b \in B) \}.
%$$

Concerning separable extensions and Hirata separable extensions, 
the followings  are well known.

%\medskip

\begin{Lemma}{\rm (\cite[Definition 2]{HS})}\label{S}
$A/B$ is separable if and only if there exists $\sum_i z_i \otimes w_i \in (A\otimes_B A)^A$
such that $\sum_i z_i w_i =1$.
\end{Lemma}

\begin{Lemma} {\rm (\cite[Proposition 1]{S1})}\label{HS}
$A/B$ is Hirata separable if and only if there exists 
$g_i \in V_A(B)$ and 
$\sum_j z_{ij} \otimes w_{ij} \in (A\otimes_BA)^A$ such that
$1 \otimes 1 = \sum_{i}g_i\sum_j z_{ij} \otimes w_{ij} = \sum_{i}\sum_j z_{ij}\otimes w_{ij} g_i$.
\end{Lemma}

Let $B$ a ring with  identity element 1, 
$\rho$  an automorphism of $B$, and $D$ a $\rho$-derivation (i.e.  
an additive endomorphism of $B$ such that $D(\alpha \beta) =D(\alpha) \beta + \rho(\alpha)D(\beta)$ 
for any $\alpha$, $\beta \in B$). 
By $B[X;\rho, D]$, we denote  the skew polynomial ring in which
the multiplication is given by  $\alpha X = X\rho(\alpha) + D(\alpha)$ for any  $\alpha \in B$. 
In particular, we write $B[X; \rho] = B[X;\rho, 0]$   and $B[X; D] = B[X; 1, D]$. 
Moreover, by $B[X;\rho, D]_{(0)}$,  we denote the set of all monic polynomials $f$ 
in $B[X;\rho, D]$ such that $fB[X;\rho, D] = B[X;\rho, D]f$. 
For any polynomial $f \in B[X;\rho, D]_{(0)}$, 
the  residue ring $B[X;\rho, D]/fB[X;\rho, D]$ is a free ring extension of $B$. 
A polynomial $f$ in  $B[X;\rho, D]_{(0)}$ is called {\it separable} (resp. {\it Hirata separable})  in $B[X;\rho, D]$ 
if   $B[X;\rho, D]/fB[X;\rho, D]$ is a separable (resp. Hirata separable) extension of $B$.

Throughout this article, 
let $f = X^m + X^{m-1}a_{m-1} + \cdots + Xa_1 + a_0$ be in $B[X;\rho,D]_{(0)}$ 
%In the following
and 
\begin{align*}
Y_0 &= X^{m-1} + X^{m-2}a_{m-1} + \cdots + Xa_2 + a_1,\\
Y_1 &= X^{m-2} + X^{m-3}a_{m-1} + \cdots + Xa_3 + a_2,\\
\cdot  &\cdot  \cdot  \cdot  \cdot \\
Y_j &= X^{m-j-1} + X^{m-j-2}a_{m-1} + \cdots + Xa_{j+2} + a_{j+1},\\
\cdot  &\cdot  \cdot  \cdot  \cdot \\
Y_{m-2} &= X + a_{m-1},\\
Y_{m-1} &= 1.
\end{align*}
In addition, we shall use the following conventions:% \par
%\par
%$A=B[X;\rho, D]/fB[X;\rho, D]$\par
\begin{align*}
A&=B[X;\rho, D]/fB[X;\rho, D]\\
x &=X + fB[X;\rho, D] \in A\\
y_j &= Y_j + fB[X;\rho, D] \ \ (0 \leq j \leq m-1)\\
B^\rho &=\{ \alpha \in B \ | \ \rho(\alpha)=\alpha \}\\
B^D&=\{ \alpha \in B \ | \ D(\alpha)=0 \}\\
B^{\rho,D} &= B^\rho \cap B^D\\
%V_A(B) 
V_0 &= V_A(B)=\{ g \in A \ | \  \alpha g = g \alpha  \ {\rm for \ any} \ \alpha \in B \}\\
%A_{\rho^{m-1}} 
V_{m-1} &= \{ h \in A \ | \ \rho^{m-1}(\alpha) h = h \alpha  \ {\rm for \ any} \ \alpha \in B \}\\
(A \otimes_B A)^A &= \{ \mu \in A \otimes_B A \ | \ z \mu  = \mu z \ {\rm for \ any} \ z \in A \}
\end{align*}

\if0
We denote $A=B[X;\rho, D]/fB[X;\rho, D]$, $x=X + fB[X;\rho, D] \in A$, and 
\begin{align*}
y_j &= Y_j + fB[X;\rho, D]\\
&= x^{m-j-1} + x^{m-j-2}a_{m-1} + \cdots + xa_{j+2} + a_{j+1} \ \ (0 \leq  j \leq m-1).
\end{align*}
Moreover, we set $B^\rho=\{ \alpha \in B \ | \ \rho(\alpha)=\alpha \}$, $B^D=\{ \alpha \in B \ | \ D(\alpha)=0 \}$, 
and $B^{\rho,D}= B^\rho \cap B^D$. 
\fi

%\medskip

In \cite{M}, Y. Miyashita  studied  separable polynomials and  Hirata separable polynomials in skew polynomial rings. 
He proved the followings. 

%\medskip

\begin{Proposition} {\rm (\cite[Theorem 1.8]{M})} \label{M1}
Let $f = X^m + X^{m-1}a_{m-1} + \cdots + Xa_1 + a_0$ be in $B[X;\rho,D]_{(0)}$. 
Then $f$ is  separable  in $B[X;\rho,D]$ if and only if 
there exists $h \in V_{m-1}$ such that 
 $\sum_{j=0}^{m-1}y_jhx^j = 1$. 
\end{Proposition}

%\medskip

\begin{Proposition} {\rm (\cite[Theorem 1.9]{M})} \label{M2}
Let $f = X^m + X^{m-1}a_{m-1} + \cdots + Xa_1 + a_0$ be in $B[X;\rho,D]_{(0)}$. 
Then $f$ is  Hirata separable  in $B[X;\rho,D]$ if and only if there exist 
$g_i  \in V_0$ and $h_i \in V_{m-1}$ such that 
$\sum_i g_ix^{m-1}h_i = 1$  and $\sum_i g_ix^kh_i = 0 \ (0 \leq k \leq m-2)$.
\end{Proposition}

%\medskip

Y. Miyashita proved aboves by making use of the theory of (*)-positively filtered rings. 
However, it seems not easy for one to comprehend his proofs. 
In the previous paper \cite{YI1}, 
the author and S. Ikehata gave direct and elementary proofs of  the above 
propositions in  $B[X;\rho]$ and  $B[X;D]$, 
respectively. 
In this paper, we shall generalize our proofs for the  skew polynomial ring $B[X;\rho,D]$.
%give an alternative proof of  Proposition \ref{M1} and Proposition \ref{M2} 
%in the skew polynomial ring  $B[X;\rho,D]$. 
%The proofs are direct and elementary. 

%\bigskip

Now we shall mention briefly some properties of the coefficients of polynomials in $B[X;\rho,D]_{(0)}$.
 % which have been obtained in {\rm \cite{I1}}.
The following lemma can be proved by a direct computation. 

\begin{Lemma} {\rm (\cite[Lemma 1.1]{I1})}\label{L1}
%If $f= X^m + X^{m-1}a_{m-1} + \cdots + Xa_1 + a_0$
%=X^m+X^{m-1}a_{m-1} + \cdots + Xa_1 + a_0 \in B[X;\rho,D]$ 
Let $f= X^m + X^{m-1}a_{m-1} + \cdots + Xa_1 + a_0$ be in $B[X;\rho,D]$. 
Then $f$ is in $B[X;\rho,D]_{(0)}$ if and only if 
 $\alpha f=f\rho^m(\alpha)$ for any $\alpha \in B$ and  $Xf=f \big( X-\rho(a_{m-1})+a_{m-1}\big)$. 
\end{Lemma}

%\medskip

%We shall state the following.

In virtue of Lemma \ref{L1}, we obtain the following. 

\begin{Lemma}\label{L2}
Assume that $\rho D = D \rho$ and 
$f= X^m + X^{m-1}a_{m-1} + \cdots + Xa_1 + a_0$ is in $B[X;\rho,D]$. 
Then $f$ is in $B[X;\rho,D]_{(0)}$ if and only if
%$f= X^m + X^{m-1}a_{m-1} + \cdots + Xa_1 + a_0$ is in $B[X;\rho,D]_{(0)}$ if and only if
\begin{description}
\item{{\rm (1)}} ${\displaystyle a_i \rho^m(\alpha)= \sum_{j=i}^{m} \binom{j}{i}\rho^{j} D^{j-i}(\alpha)a_{j}}$
  \ \ $(\alpha \in B, \ 0 \leq i \leq m-1, \ a_m=1)$
\item{{\rm (2)}} $D(a_i)= a_{i-1}-\rho({a_{i-1}})-a_i \big( \rho(a_{m-1})-a_{m-1} \big)$ \ \ $(1 \leq i \leq m-1)$
\item{{\rm (3)}} $D(a_0)=a_0 \big( \rho(a_{m-1})-a_{m-1} \big)$
\end{description}
%In particular, If $f$ is in  $B[X;\rho,D]_{(0)}$ then $\beta a_i = a_i \beta$ for any  $\beta \in B^{\rho,D}$. 
\end{Lemma}

{\it Proof. } 
Let $\alpha$ be  arbitrary element in $B$. 
By Lemma \ref{L1}, $f$ is in $B[X;\rho,D]_{(0)}$ if and only if 
$\alpha f=f\rho^m(\alpha)$  and  $Xf=f(X-\rho(a_{m-1})+a_{m-1})$. 
It is easy to see that $\alpha X^j = \sum_{i=0}^{j}\binom{j}{i} X^i \rho^{j} D^{j-i}(\alpha) a_j$ 
($j \geq 0$) by an induction. 
Hence we obtain 
\begin{align*}
%\alpha f =
\sum_{j=0}^{m} \alpha X^j a_j 
&=\sum_{j=0}^m \left( \sum_{i=0}^{j}\binom{j}{i} X^i \rho^{j} D^{j-i}(\alpha) \right) a_j\\
&=\sum_{i=0}^m X^i \left(  \sum_{j=i}^m \binom{j}{i} \rho^{j} D^{j-i}(\alpha) a_j  \right).
\end{align*}
%By Comparing the coefficients of the both sides, 
This means that $\alpha f=f\rho^m(\alpha)$  implies (1), and conversely. 
By a direct computation, we have 
\begin{align*}
&f \big( X-\rho(a_{m-1})+a_{m-1} \big)\\
&=  \sum_{i=0}^{m-1} X^{i} a_i X - \sum_{i=0}^{m-1}X^i a_i \big( \rho(a_{m-1})-a_{m-1}\big)\\
&= \sum_{i=0}^{m-1} X^{i} \big( X\rho(a_i) +D(a_i) \big)
- \sum_{i=0}^{m-1}X^i a_i \big( \rho(a_{m-1})-a_{m-1} \big) \\
&= \sum_{i=0}^{m-1} X^{i+1} \rho(a_i) + \sum_{i=0}^{m-1}X^i \Big( D(a_i)-a_i \big( \rho(a_{m-1})-a_{m-1} \big) \Big)\\
&= X^m a_{m-1} + \sum_{i=1}^{m-1} X^i \Big( \rho(a_{i-1}) + D(a_i) -a_i \big( \rho(a_{m-1})-a_{m-1} \big)   \Big)\\
& \ \ \ \, + D(a_0) -a_0 \big( \rho(a_{m-1})-a_{m-1} \big).
\end{align*}
This means that  $Xf=f\big( X-\rho(a_{m-1})+a_{m-1} \big)$ implies (2) and (3), and conversely. 
This completes the proof.

%\bigskip
%By Lemma \ref{L2}, we can easily see the following.
%In virtue of Lemma \ref{L2}, we obtain the following. 

\begin{Corollary}\label{R}
Assume that $\rho D = D \rho$ and 
 $f= X^m + X^{m-1}a_{m-1} + \cdots + Xa_1 + a_0$ is in $B[X;\rho,D]_{(0)} \cap B^\rho[X]$. 
Then $f$ is in $C(B^{\rho,D})[X]$, where $C(B^{\rho,D})$ is the center of $B^{\rho,D}$. 
Moreover, 
\[
\alpha a_i = \sum_{j=i}^m (-1)^{j-i} \binom{j}{i} a_j \rho^{m-j}D^{j-i}(\alpha) \ \ (\alpha \in B,  \, 0 \leq j \leq m, \, a_m = 1)
\]
\end{Corollary}

{\it Proof. } Let $f= X^m + X^{m-1}a_{m-1} + \cdots + Xa_1 + a_0$ be in $B[X;\rho,D]_{(0)} \cap B^\rho[X]$. 
Then $a_i \in B^\rho$ ($0 \leq i \leq m-1$) implies $a_i \in B^D$ ($0 \leq i \leq m-1$) by Lemma \ref{L2} (2) and (3). 
Therefore $a_i \in C(B^{\rho,D})$ ($0 \leq i \leq m-1$) by Lemma \ref{L2} (1). 
Let $\alpha$ be arbitrary element in $B$. 
By an  easy induction,  we see that 
$
X^{j}\alpha=\sum_{i=0}^{j} (-1)^{j-i} \binom{j}{i}\rho^{-j}D^{j-i}(\alpha)X^{i} \ \ (j \geq 0).
$
Hence we obtain 
\begin{align*}
%\sum_{j=0}^{m} X^j a_j \alpha = 
\sum_{j=0}^{m} a_j X^j  \rho^{m}(\alpha) 
&= \sum_{j=0}^{m} a_j \left( \sum_{i=0}^{j} (-1)^{j-i} \binom{j}{i}\rho^{m-j}D^{j-i}(\alpha)X^{i} \right)\\
&= \sum_{i=0}^{m} \left(  \sum_{j=i}^{m} (-1)^{j-i} \binom{j}{i} a_j \rho^{m-j}D^{j-i}(\alpha) \right)X^{i}
\end{align*}
Therefore $\alpha f = f \rho^m(\alpha)$ implies our assertion. 
This completes the proof.

\section{Main Result}
The conventions and notations employed in the preceding section 
will be used in this section. 
We assume that $\rho D= D\rho$ and 
$f= X^m + X^{m-1}a_{m-1} + \cdots + Xa_1 + a_0$ is in $B[X;\rho,D]_{(0)} \cap B^\rho[X]$. 
Note that $f$ is in $C(B^{\rho,D})[X]$ by Corollary \ref{R}.  
First we shall prove the following.

\begin{Lemma} \label{L}
\[
(A\otimes_BA)^A = \left\{ \left. \sum_{j=0}^{m-1}y_jh\otimes x^j \,  \right|  
\, h \in V_{m-1} \right\}.
\]
\end{Lemma}

{\it Proof. }  %The proof of this lemma is the same as the previous one. 
%Then we outline the proof. 
%See the previous one to check the detail of computations in the proof.
Since $\{ 1, x, x^2, \cdots , x^{m-1} \}$ is a free $B$-basis for $A$, 
every element in $A\otimes_BA$ has a form $\sum_{j=0}^{m-1}z_j\otimes x^j$ for some 
$z_j \in A$. 
Let $\sum_{j=0}^{m-1}z_j\otimes x^j$ be in  $(A\otimes_BA)^A$ 
and $\alpha$ arbitrary element  in $B$. 
Since $\alpha\left( \sum_{j=0}^{m-1}z_j\otimes x^j \right) = \left( \sum_{j=0}^{m-1}z_j\otimes x^j \right)\alpha$ 
and  $x^{j}\alpha=\sum_{i=0}^{j} (-1)^{j-i} \binom{j}{i}\rho^{-j}D^{j-i}(\alpha)x^{i}$ ($j \geq 0$), 
we have
\begin{align*}
\sum_{j=0}^{m-1}\alpha z_j\otimes x^j
&= \sum_{j=0}^{m-1} z_j\otimes x^j\alpha\\
&= \sum_{j=0}^{m-1} z_j\otimes \left(\sum_{i=0}^{j} (-1)^{j-i} \binom{j}{i} \rho^{-j}D^{j-i}(\alpha )x^{i}\right)\\
&= \sum_{j=0}^{m-1} z_j \Big(\sum_{i=0}^{j} (-1)^{j-i} \binom{j}{i}\rho^{-j}D^{j-i}(\alpha ) \Big)\otimes x^{i}\\
&= \sum_{i=0}^{m-1} \left( \sum_{j=i}^{m-1} (-1)^{j-i} \binom{j}{i}z_j\rho^{-j}D^{j-i}(\alpha ) \right) \otimes x^{i}\\
&= \sum_{j=0}^{m-1} \left( \sum_{i=j}^{m-1} (-1)^{i-j} \binom{i}{j} z_i\rho^{-i}D^{i-j}(\alpha ) \right) \otimes x^{j}.
\end{align*}
Hence we obtain %$\alpha z_j= \sum_{i=j}^{m-1} \binom{i}{j}(-1)^{i-j}z_i\rho^{-i}D^{i-j}(\alpha)$
\[
\alpha z_j= \sum_{i=j}^{m-1} (-1)^{i-j} \binom{i}{j}z_i\rho^{-i}D^{i-j}(\alpha ) \ \ (0 \leq j \leq m-1).
\]
Now we let $h=z_{m-1}$. 
Obviously, $h$ is in $V_{m-1}$. 
%$\alpha h = h \rho^{-(m-1)}(\alpha)$, namely
%$$
%\rho^{m-1}(\alpha) h = h \alpha \ \ (\alpha \in B).
%$$
Since $x \left( \sum_{j=0}^{m-1}z_j\otimes x^j \right) = \left( \sum_{j=0}^{m-1}z_j\otimes x^j \right) x$ and 
$x^m = - \sum_{j=0}^{m-1}x^ja_j = - \sum_{j=0}^{m-1}a_jx^j$, 
we see that 
\begin{align*}
\sum_{j=0}^{m-1}xz_j\otimes x^j 
&= \sum_{j=0}^{m-1}z_j\otimes x^{j+1}\\
&= \sum_{j=0}^{m-2}z_j\otimes x^{j+1}+h\otimes x^m\\
&= \sum_{j=1}^{m-1}z_{j-1}\otimes x^{j}+h\otimes \left( -\sum_{j=0}^{m-1}a_jx^j \right)\\
&= \sum_{j=1}^{m-1}z_{j-1}\otimes x^j-\sum_{j=0}^{m-1}ha_j\otimes x^j\\
%&= \sum_{j=1}^{m-1}(z_{j-1}-z_{m-1}a_j)\otimes x^j-z_{m-1}a_0\otimes 1
&= - ha_0\otimes 1 + \sum_{j=1}^{m-1}(z_{j-1}-ha_j)\otimes x^j.
\end{align*}
It follows that 
$xz_j = z_{j-1}-ha_j$  ($1 \leq j \leq m-1$) and  $xz_0  = -ha_0$.
Noting that $h a_j = \rho^{m-1}(a_j)h = a_j h$ $(0 \leq j \leq m-1)$, we have
$z_{j-1} = x z_j + a_j h$  ($1 \leq j \leq m-1$) and  $x z_0 = -a_0 h$. 
This implies $z_j = y_jh$  ($0 \leq j \leq m-1$), inductively.

Conversely, assume that $h$ is in $V_{m-1}$.  
To show $ \sum_{j=0}^{m-1}y_jh\otimes x^j \in (A \otimes_BA)^A$, 
it is suffice to prove that 
$
x \left( \sum_{j=0}^{m-1}y_jh\otimes x^j \right)
= \left( \sum_{j=0}^{m-1}y_jh\otimes x^j \right)x$ and  
$\alpha \left( \sum_{j=0}^{m-1}y_jh\otimes x^j \right)
= \left( \sum_{j=0}^{m-1}y_jh\otimes x^j \right)\alpha$ 
for any $\alpha \in B$. 
Noting that $xy_j  = y_{j-1} - a_j \  (1 \leq j \leq m-1)$ and  $xy_0  = -a_0$, 
we obtain 
\begin{align*}
x \left( \sum_{j=0}^{m-1}y_jh\otimes x^j \right)
%&= \sum_{j=0}^{m-1}xy_jh\otimes x^j\\
&= xy_0h\otimes 1 + \sum_{j=1}^{m-1}xy_j h\otimes x^j\\
&= (-a_0) h\otimes 1 + \sum_{j=1}^{m-1}(y_{j-1}-a_j)h\otimes x^j\\
&= h \otimes (-a_0)+ \sum_{j=1}^{m-1} (-a_j)h\otimes x^j + \sum_{j=1}^{m-1}y_{j-1}h\otimes x^j   \\
&= h \otimes (-a_0) + h\otimes \left( -\sum_{j=1}^{m-1}x^ja_j \right) + \sum_{j=0}^{m-2}y_{j}h\otimes x^{j+1} \\
%&= h \otimes (-\sum_{j=0}^{m-1}x^ja_j) + \sum_{j=0}^{m-2}y_{j}h\otimes x^{j+1}\\
&= h\otimes x^m + \sum_{j=0}^{m-2}y_{j}h\otimes x^{j+1}\\
&= \left( \sum_{j=0}^{m-1}y_jh\otimes x^j \right) x.
\end{align*}
Next, for any $\alpha \in B$, we have 
%$\alpha \left( \sum_{j=0}^{m-1}y_jh\otimes x^j \right)  
%= \sum_{j=0}^{m-1}\alpha y_jh\otimes x^j$ and
\begin{align*}
%\alpha(\sum_{j=0}^{m-1}y_jh\otimes x^j) & = \sum_{j=0}^{m-1}\alpha y_jh\otimes x^j, \ {\rm and }\\
%\end{align*}
%\begin{align*}
\left( \sum_{j=0}^{m-1}y_jh\otimes x^j \right) \alpha 
%& = \sum_{j=0}^{m-1}y_jh\otimes x^j\alpha \\
& = \sum_{j=0}^{m-1}y_jh\otimes \left( \sum_{i=0}^j (-1)^{j-i} \binom{j}{i}\rho^{-j}D^{j-i}(\alpha)x^i \right)  \\
&= \sum_{j=0}^{m-1} y_jh \left( \sum_{i=0}^{j} (-1)^{j-i} \binom{j}{i}\rho^{-j}D^{j-i}(\alpha) \right) \otimes x^i \\
& = \sum_{i=0}^{m-1} \left( \sum_{j=i}^{m-1} (-1)^{j-i} \binom{j}{i} y_j \rho^{m-j-1}D^{j-i}(\alpha) \right) h \otimes x^i \\
%& = \sum_{i=0}^{m-1} \Big(\sum_{j=i}^{m-1}y_j\binom{j}{i}(-1)^{j-i}\rho^{m-j-1}D^{j-i}(\alpha)  h \Big)\otimes x^i \\
& = \sum_{j=0}^{m-1} \left( \sum_{i=j}^{m-1}(-1)^{i-j} \binom{i}{j} y_i \rho^{m-i-1}D^{i-j}(\alpha)  \right)h \otimes x^j.
\end{align*}
Hence to show 
$\alpha \left( \sum_{j=0}^{m-1}y_jh\otimes x^j \right) = 
\left(\sum_{j=0}^{m-1}y_jh\otimes x^j \right)\alpha$,
it is suffices to prove that 
\begin{align*}
\alpha y_j = \sum_{i=j}^{m-1}(-1)^{i-j} \binom{i}{j} y_i \rho^{m-i-1}D^{i-j}(\alpha) \ \ (0 \leq j \leq m-1).
\end{align*}  
We shall show it by induction. Since $y_{m-1} = 1$, it is true when $j = m-1$. For some 
$j \ (0 \leq j \leq m-2)$, we assume that 
\begin{align*}
\alpha y_{j+1} = \sum_{i=j+1}^{m-1} (-1)^{i-j-1}\binom{i}{j+1} y_i \rho^{m-i-1}D^{i-j-1}(\alpha).
\end{align*}
Noting that
$\alpha a_j = \sum_{i=j}^m (-1)^{i-j} \binom{i}{j} a_i \rho^{m-i}D^{i-j}(\alpha)$ ($0 \leq j \leq m,$, $a_m = 1$)
and $xy_{i} =y_{i-1} - a_{i} \ (1 \leq i \leq m-1)$, 
we have
\begin{align*}
\alpha  y_j &=  \alpha xy_{j+1}+\alpha a_{j+1}\\
&=   x\rho(\alpha) y_{j+1}+D(\alpha)y_{j+1} +\alpha a_{j+1}\\
 &= x \left( \sum_{i=j+1}^{m-1} (-1)^{i-j-1} \binom{i}{j+1} y_i\rho^{m-i-1}D^{i-j-1}( \rho(\alpha)) \right)\\
 & \ \ \   \ \ +\sum_{i=j+1}^{m-1}(-1)^{i-j-1}\binom{i}{j+1} y_i \rho^{m-i-1}D^{i-j-1}(D(\alpha))\\
 & \ \ \   \ \ +\sum_{i=j+1}^m (-1)^{i-j-1} \binom{i}{j+1} a_i \rho^{m-i}D^{i-j-1}(\alpha)\\
&=\sum_{i=j+1}^{m-1} (-1)^{i-j-1} \binom{i}{j+1} (y_{i-1}-a_i)\rho^{m-i}D^{i-j-1}(\alpha)\\
  &  \ \ \ \ \  -\sum_{i=j+1}^{m-1} (-1)^{i-j} \binom{i}{j+1} y_i \rho^{m-i-1}D^{i-j}(\alpha)\\
  &   \ \ \ \ \ +\sum_{i=j+1}^m (-1)^{i-j-1} \binom{i}{j+1} a_i \rho^{m-i}D^{i-j-1}(\alpha)\\
&=\sum_{i=j+1}^{m-1} (-1)^{i-j-1} \binom{i}{j+1} y_{i-1}\rho^{m-i}D^{i-j-1}(\alpha) \\
 & \ \ \ \ \ - \sum_{i=j+1}^{m-1} (-1)^{i-j-1} \binom{i}{j+1} a_i\rho^{m-i}D^{i-j-1}(\alpha)\\
  &  \ \ \ \ \  -\sum_{i=j+1}^{m-1} (-1)^{i-j} \binom{i}{j+1} y_i \rho^{m-i-1}D^{i-j}(\alpha)\\
  &   \ \ \ \ \ +\sum_{i=j+1}^{m}  (-1)^{i-j-1} \binom{i}{j+1} a_i \rho^{m-i}D^{i-j-1}(\alpha) \\
%+(-1)^{m-j-1} \binom{m}{j+1}D^{m-j-1}(\alpha) \\
&= \sum_{i=j}^{m-2} (-1)^{i-j} \binom{i+1}{j+1} y_{i} \rho^{m-i-1}D^{i-j}(\alpha)\\
& \ \ \ \ \ -\sum_{i=j+1}^{m-1}(-1)^{i-j}\binom{i}{j+1} y_i \rho^{m-i-1}D^{i-j}(\alpha)\\
& \ \  \ \ \ + (-1)^{m-j-1}\binom{m}{j+1}D^{m-j-1}(\alpha)\\
&= y_j\rho^{m-j-1}(\alpha)\\
& \ \ \ \ \ + \sum_{i=j+1}^{m-2} (-1)^{i-j}  \left\{ \binom{i+1}{j+1} - \binom{i}{j+1}\right\} y_{i} \rho^{m-i-1}D^{i-j}(\alpha)\\
& \ \ \ \ \ + (-1)^{m-j-1} \left\{ \binom{m}{j+1}-\binom{m-1}{j+1} \right\} D^{m-j-1}(\alpha) \\
%& -y_{m-1}\binom{m-1}{j+1}(-1)^{m-j-1}D^{m-j-1}\\
%&+\binom{m}{j+1}(-1)^{m-j-1}D^{m-j-1}(\alpha)\\
%
%= &y_j\rho^{m-j-1}(\alpha) + \sum_{k=j+1}^{m-2}y_{k}\binom{k}{j}(-1)^{k-j}D^{k-j}(\alpha)\\
%&+y_{m-1}\bigg{\{}\binom{m}{j+1}-\binom{m-1}{j+1}\bigg{\}}(-1)^{m-j-1}D^{m-j-1}(\alpha) \\
%\end{eqnarray*}
%\begin{eqnarray*}
&= y_j\rho^{m-j-1}(\alpha)\\
& \ \ \ \ \  + \sum_{i=j+1}^{m-2} (-1)^{i-j} \binom{i}{j} y_{i} \rho^{m-i-1}D^{i-j}(\alpha)\\
& \ \ \  \ \  + (-1)^{m-j-1} \binom{m-1}{j}D^{m-j-1}(\alpha)\\
%&= y_j\rho^{m-j-1}(\alpha) + \sum_{k=j+1}^{m-1}y_{k}\binom{k}{j}(-1)^{k-j}\rho^{m-k-1}D^{k-j}(\alpha)\\
&=  \sum_{i=j}^{m-1}(-1)^{i-j}\binom{i}{j} y_{i} \rho^{m-i-1} D^{i-j}(\alpha).
\end{align*}
This completes the proof of Lemma 2.1.

\bigskip

\bigskip

Finally we shall prove  Proposition \ref{M1} and Proposition \ref{M2}. 

\bigskip

{\it Proof of Proposition \ref{M1}. }It is obvious by Lemma \ref{S} and Lemma \ref{L}. 

\bigskip

{\it Proof of Proposition \ref{M2}. }Let $f$ be  Hirata separable  in $B[X;\rho,D]$. 
It follows from Lemma \ref{HS} and Lemma \ref{L}  that there exist $g_i \in V_0$ 
and %$h_i \in V_{m-1}$ 
$\sum_{j=0}^{m-1}y_j h_i \otimes x^j \in (A\otimes_BA)^A$  with $h_i \in V_{m-1}$ 
such that 
\[
1 \otimes 1 = \sum_i g_i \sum_{j=0}^{m-1}y_j h_i \otimes x^j
= \sum_{j=0}^{m-1} \left(\sum_i g_i y_jh_i \right)\otimes x^j.
\]   
This implies 
\[
\sum_i g_i y_0 h_i =1 \ \  {\rm and}  \ \ \sum_i g_i y_k h_i =0 \ (1 \leq k \leq m-1).
\]
Then we obtain inductively,
\[
\sum_i g_i x^k h_i = 0 \ (0 \leq k \leq m-2)  \ \ {\rm and} \ \
\sum_i g_i x^{m-1} h_i = 1.
\]

Conversely, let  $g_i$ be in $V_0$ and  $h_i $  in $V_{m-1}$ such that 
$\sum_i g_ix^{m-1}h_i = 1$ and $\sum_i g_ix^kh_i = 0 \ \ (0 \leq k \leq m-2)$.   
Note that $\sum_{j=0}^{m-1} y_j h_i\otimes x^j \in (A \otimes_B A)^A$ 
by Lemma \ref{L}. 
By an easy induction, we see that 
\[
\sum_i g_i y_k h_i =0 \ (1 \leq k \leq m-1) \ \ {\rm and} \ \ \sum_i g_i y_{0} h_i =1.
\]
This implies 
\[
\sum_i g_i\sum_{j=0}^{m-1} y_j h_i\otimes x^j
= \sum_{j=0}^{m-1} \left(\sum_i g_i y_j h_i \right) \otimes x^j
= 1\otimes 1.
\]
Therefore,    
$f$ is  Hirata separable in $B[X;\rho,D]$ by Lemma \ref{HS}. 
 This complets the proof of Proposition \ref{M2}.

\bigskip

\bigskip

{\bf Acknowledgement.} The author would like to thank the referee for his valuable comments and suggestions.


\begin{thebibliography}{99}

\bibitem{HS}{ K. Hirata and K. Sugano,
On semisimple extensions and separable extensions over noncommutative rings,
\em J. Math. Soc. Japan }{\bf 18} (1966), 360--373. 

\bibitem {I1}{ S. Ikehata,
On separable polynomials and Frobenius polynomials in skew
polynomial rings,
\em Math. J. Okayama Univ.,} {\bf 22} 1980,   115--129.

\bibitem {I2}{ S. Ikehata,
Azumaya algebras and skew polynomial rings,
\em Math, J. Okayama Univ.,} {\bf 23}  1981,  19--32.

\bibitem {I5}{ S. Ikehata,
A note on separable polynomials in skew polynomial rings of
derivation type,
\em Math. J. Okayama Univ.,} {\bf 22} 1980,  59--60.

\bibitem {I8}{ S. Ikehata,
On $H$-separable polynomials of prime degree,
\em Math. J. Okayama Univ.,} {\bf 33}  1991,   21--26.

%\bibitem {I9}{ S. Ikehata,
%Purely inseparable ring extensions and $H$-separable polynomials,
%\em Math. J. Okayama Univ.,}  {\bf 40} 1998,   55--63.

%\bibitem {I10}{ S. Ikehata,
%Purely inseparable ring extensions and Azumaya algebras,
%\em Math. J. Okayama Univ.,}  {\bf 41} 1999,   63--69.

\bibitem {I11}{ S. Ikehata,
On $H$-separable and Galois polynomials of degree $p$ in skew polynomial rings, 
\em Int. Math. Forum,}  {\bf 3} 2008, no. 29-32, 1581-1586. 

\bibitem {I12}{ S. Ikehata,
On separable and $H$-separable polynomials of degree $p$ in skew polynomial rings,
\em Int. J. Pure Appl. Math.,} {\bf \ 51} 2009, no. 1, 149--156.

%\bibitem {J}{ N. Jacobson, 
%Lectures in abstract algebra. Vol III: Theory of fields and Galois theory, 
%\em D. Van Nostrand Co., Inc.,} 1964

\bibitem {K1}{ K. Kishimoto,
On abelian extensions of rings. I,
\em Math. J. Okayama Univ.,}  {\bf 14} 1970,   159--174.

\bibitem {M}{ Y. Miyashita,
On a skew polynomial ring,
\em J. Math. Soc. Japan,}
{\bf 31} 1979,  no. 2,   317--330.

\bibitem {N1}{ T. Nagahara,
On separable polynomials of degree $2$ in skew polynomial rings,
\em Math. J. Okayama Univ.,}  {\bf 19} 1976,   65--95.

\bibitem {N2}{ T. Nagahara, 
A note on separable polynomials in skew polynomial rings of atutomorphism type, 
\em Math. J. Okayama Univ.,}  {\bf 22} 1980,  73--76. 

%\bibitem {N3}{ T. Nagahara,
%A note on imbeddings of noncommutative separable extensions in Galois extensions,
%\em Houston J. Math.,} {\bf 12} 1986,   411--417.

\bibitem {OS}{ H. Okamoto and S. Ikehata,
On $H$-separable polynomials of degree $2$,
\em Math. J. Okayama Univ.,} {\bf 32} 1990,  53--59.

\bibitem {S1}{ K. Sugano,
Separable extensions and Frobenius extensions,
\em Osaka J. Math. }{\bf 7} (1970), 29--40.

%\bibitem {S2}{ K. Sugano, 
%Note on cyclic Galois extensions,
%\em Proc. Japan Acad.,} {\bf 57}, Ser. A 1981, 60--63.

\bibitem {SX1}{ G. Szeto and L. Xue,
On the Ikehata theorem for $H$-separable skew polynomial rings,
\em Math. J. Okayama Univ.,} {\bf 40} 1998,  27--32.

\bibitem {YI1} {S. Yamanaka and S. Ikehata, 
An alternative proof of Miyasita's theorem in a skew polynomial rings,
\em Int. J. Algebra,} {\bf  21} 2012, 1011--1023

%\bibitem{YI2} {S. Yamanaka and S. Ikehata, 
%On Galois polynomials of degree $p$ in skew polynomial rings of derivation type, 
%\em Southeast Asian Bull. Math.,} {\bf 37} 2013, 625--634.


\end{thebibliography}
\end{document}